\documentclass[12pt,reqno]{amsart}
\usepackage[a4paper, hmargin={2.7cm,2.7cm},vmargin={3.3cm,3.3cm}]{geometry}
\usepackage[english]{babel}
\usepackage{color}
\usepackage[noadjust]{cite}
\usepackage{amsmath}
\usepackage{amssymb}
\usepackage{amsfonts}
\usepackage{amsthm}
\usepackage{enumerate}
\usepackage{amsthm}
\usepackage{dsfont}
\usepackage{xcolor}

\usepackage{mathtools}
\mathtoolsset{showonlyrefs}

\usepackage[textwidth=20mm]{todonotes}
\usepackage[hypertexnames=false]{hyperref}
\usepackage[nameinlink]{cleveref}
\usepackage{commath}
\usepackage{esint}
\usepackage{etoolbox}

\usepackage{amssymb}   % Extra symbols
\usepackage{amsthm}    % Theorem-like environments
\usepackage{thmtools}  % Theorem-like environments
\usepackage{mathtools} % Fonts and environments for mathematical formuale
\usepackage{mathrsfs}  % Scfript font with \mathscr{}
\usepackage{commath}
\usepackage{bbm}
\usepackage[textwidth=20mm]{todonotes}

\theoremstyle{plain}
\newtheorem{theorem}{Theorem}[section]
\newtheorem{lemma}[theorem]{Lemma}
\newtheorem{proposition}[theorem]{Proposition}

\newtheorem{conjecture}[theorem]{Conjecture}

\theoremstyle{definition}

\theoremstyle{remark}

\numberwithin{equation}{section}

\usepackage[all]{xy}

\makeatletter
\providecommand*{\cupdot}{%
  \mathbin{%
    \mathpalette\@cupdot{}%
  }%
}
\newcommand*{\@cupdot}[2]{%
  \ooalign{%
    $\m@th#1\cup$\cr
    \sbox0{$#1\cup$}%
    \dimen@=\ht0 %
    \sbox0{$\m@th#1\cdot$}%
    \advance\dimen@ by -\ht0 %
    \dimen@=.5\dimen@
    \hidewidth\raise\dimen@\box0\hidewidth
  }%
}

\providecommand*{\bigcupdot}{%
  \mathop{%
    \vphantom{\bigcup}%
    \mathpalette\@bigcupdot{}%
  }%
}
\newcommand*{\@bigcupdot}[2]{%
  \ooalign{%
    $\m@th#1\bigcup$\cr
    \sbox0{$#1\bigcup$}%
    \dimen@=\ht0 %
    \advance\dimen@ by -\dp0 %
    \sbox0{\scalebox{2}{$\m@th#1\cdot$}}%
    \advance\dimen@ by -\ht0 %
    \dimen@=.5\dimen@
    \hidewidth\raise\dimen@\box0\hidewidth
  }%
}
\makeatother

\newcommand{\N}{\mathbb{N}}
\newcommand{\Z}{\mathbb{Z}}

\newcommand{\R}{\mathbb{R}}

\DeclareMathOperator{\diag}{diag}

\title[Wavelet orthonormal bases for nonexpansive dilations]{Wavelet orthonormal  bases \\ for nonexpansive dilations}

\author{Jordy Timo van Velthoven}
\address{Faculty of Mathematics,
University of Vienna, 
Oskar-Morgenstern-Platz 1,
1090 Vienna, Austria}
\email{jordy-timo.van-velthoven@univie.ac.at}

\author{Felix Voigtlaender}
\address{
Mathematical Institute for Machine Learning and Data Science (MIDS),
Catholic University of Eichstätt–Ingolstadt (KU),
Auf der Schanz 49, 85049 Ingolstadt, Germany
}

\email{felix.voigtlaender@ku.de}



\begin{document}

\maketitle

\begin{abstract}
We construct wavelet orthonormal bases for  nonexpansive dilations and integer translations that do not admit a wavelet set. In addition, these orthonormal bases do not satisfy the Calder\'on sum formula. In particular, we  show the existence of a wavelet basis for a dilation matrix with determinant one.
\end{abstract}

\section{Introduction}
For matrices $A, P \in \mathrm{GL}(d, \R)$ and a function $\psi \in L^2 (\R^d)$, the associated wavelet system is 
\begin{align} \label{eq:discrete_wavelet}
\big\{ |\det(A)|^{-j/2} \psi (A^{-j} \cdot -  \gamma) \big \}_{j \in \mathbb{Z}, \gamma \in P \mathbb{Z}^d}.
\end{align}
A common assumption in the study of such wavelet systems is that the dilation matrix is expansive, i.e., all its eigenvalues have absolute value strictly greater than one. Under this assumption, characterizations of orthonormal bases and general Parseval frames have been obtained in, e.g., \cite{chui2002characterization, calogero2000characterization, bownik2000characterization, rzeszotnik2001calderon}. More generally, such characterizations were shown to hold for dilation matrices that are expanding on a subspace in  \cite{hernandez2002unified}, amplifying dilations \cite{laugesen2002translational}, or dilations satisfying the lattice counting estimate \cite{bownik2017wavelets}. 
However, for general dilation matrices, characterizations of wavelet systems forming an orthonormal basis or Parseval frame are not known. 
In particular, this has led to the following conjecture, which was implicitly raised in \cite[p. 177]{speegle2003existence} and \cite{bownik2004spectral}, and explicitly stated as \cite[Conjecture 1]{bownik2017wavelets} and as the open problem \cite[Problem 3.3]{bownik2020open}.

\begin{conjecture}[\cite{bownik2017wavelets, bownik2020open, bownik2004spectral, speegle2003existence}] \label{conj:wavelet}
Let $A, P \in \mathrm{GL}(d, \mathbb{R})$ and $\psi \in L^2 (\mathbb{R}^d)$. Suppose that \[ \{|\det(A)|^{-j/2} \psi (A^{-j} \cdot - \gamma) \}_{j \in \mathbb{Z}, \gamma \in P \mathbb{Z}^d}\]
is an orthonormal basis, or more generally a Parseval frame, for $L^2 (\mathbb{R}^d)$. Then the Calder\'on sum formula holds:
\begin{align} \label{eq:calderon0}
\sum_{j \in \mathbb{Z}} |\widehat{\psi} ((A^t)^j \xi)|^2 = |\det(P)| \quad \text{for a.e.} \quad \xi \in \mathbb{R}^d.
\end{align}
\end{conjecture}

The identity \eqref{eq:calderon0} is part of the known characterizations of wavelet systems forming orthonormal bases or Parseval frames under the aforementioned assumptions on the dilation matrices. A solution of Conjecture~\ref{conj:wavelet} would therefore be a substantial step
in the understanding of general wavelet systems forming orthonormal bases or Parseval frames. It was recently shown in \cite{enstad2026calderon} that Conjecture~\ref{conj:wavelet} is true under an additional localization assumption on the wavelet function.

Although the understanding of general wavelet orthonormal bases is far from complete, the recent breakthrough \cite{bownik2026simultaneous} yielded a complete characterization of the matrices $B \in \mathrm{GL}(d, \R)$ and full-rank lattices $\Lambda \subseteq \R^d$ admitting a wavelet set; see also \cite{ionascu2006simultaneous} for the characterization in dimension two. Here, we recall that a measurable set $W \subseteq \R^d$ is a \emph{wavelet set} for the pair $(B, \Lambda)$ if both
\[
\big\{ B^j W : j \in \Z \big\} \quad \text{and} \quad \big\{ W + \lambda : \lambda \in \Lambda \big\}
\]
are measurable tilings of $\R^d$ up to null sets. If $W \subseteq \R^d$ is a wavelet set for $(B, \Lambda)$, then the function $\psi \in L^2 (\R^d)$ with Fourier transform $\widehat{\psi} = |W|^{-1/2} \cdot \mathds{1}_W$ defines a wavelet orthonormal basis for $L^2 (\R^d)$ of the form \eqref{eq:discrete_wavelet} for the transpose matrix $A:= B^t$ of $B$ and the dual lattice $\Gamma := \Lambda^*$ of $\Lambda$, see, e.g., \cite{wang2002wavelets, dai1998wandering}. In \cite[Conjecture 1.2]{bownik2026simultaneous}, the following conjecture on the relation between the existence of general orthonormal bases and wavelet sets was made; see also \cite{larson2007unitary, bownik2026msf} for stronger forms of this conjecture.

\begin{conjecture}[\cite{bownik2026simultaneous}] \label{con:waveletset}
    For each pair $(A, \Gamma)$ for which there exists a wavelet orthonormal basis $
\{ |\det(A)|^{-j/2} \psi (A^{-j} \cdot -  \gamma) \big \}_{j \in \mathbb{Z}, \gamma \in \Gamma}$ for $L^2 (\R^d)$, there exists a $(B, \Lambda)$ wavelet set, where $B$ is the transpose of $A$ and $\Lambda$ the dual lattice of $\Gamma$.
\end{conjecture}

In the present paper, we construct wavelet orthonormal bases for which the Calderon sum formula \eqref{eq:calderon0} fails and for whose dilation matrix there does not exist an associated wavelet set, showing that both Conjecture~\ref{conj:wavelet} and Conjecture~\ref{con:waveletset} are false. Our main result is the following theorem.

\begin{theorem} \label{thm:main}
Let $a \geq 1$ and $A_a := \diag(2a, 1/2 )$. There exists $\psi^{(a)} \in L^2 (\R^2)$ such that the wavelet system
$
\big\{ a^{-j/2} \psi^{(a)} (A_a^{-j} \cdot -  \gamma) \big \}_{j \in \mathbb{Z}, \gamma \in  \mathbb{Z}^2}
$
is an orthonormal basis for $L^2 (\R^2)$ satisfying
\begin{align} \label{eq:calderon}
\sum_{j \in \mathbb{Z}} |\widehat{\psi^{(a)}} ((A_a^t)^j \xi)|^2 \leq \frac{1}{3} \quad \text{for a.e.} \quad \xi \in \mathbb{R}^2.
\end{align}
The pair $(A_a^t, \Z^2) = (A_a, \Z^2)$ does not admit a wavelet set.
\end{theorem}

In particular, Theorem \ref{thm:main} shows the existence of a wavelet orthonormal basis for the dilation matrix $A = \diag(2, 1/2)$. The question whether an orthonormal basis for a dilation matrix of determinant one exists was mentioned as an open problem in \cite{bownik2017wavelets}. It is well-known that for such a matrix the Calder\'on  
sum formula \eqref{eq:calderon0} cannot hold (cf. \cite[Lemma 3.5]{bownik2017wavelets}) and that there cannot exist a wavelet set (cf. \cite[Section 3]{laugesen2002characterization} or \cite[Theorem 4]{larson2006explicit}). The fact that none of the pairs in Theorem \ref{thm:main} admits a wavelet
set follows from the characterizations in \cite{ionascu2006simultaneous, bownik2026simultaneous}; see Proposition \ref{prop:no_wavelet_set} for details.

The construction of the orthonormal bases of Theorem~\ref{thm:main} starts with an orthonormal basis for $L^2 (\R)$ of the form $\{g_n(\cdot - 2^n \ell) :n\geq2,\ \ell \in\Z\}$.
This is obtained by an adaptation of the
construction of Bownik and Rzeszotnik
\cite[Example~3.2]{bownik2004spectral},
which showed that a Calder\'on-type identity cannot in general hold for a generalized shift-invariant system as studied in \cite{hernandez2002unified}. In order to obtain the wavelet basis for $L^2 (\R^2)$,
we then place dilates of the functions $g_n$ on disjoint intervals in the
first partial Fourier variable. An elementary Fourier analytic argument allows us to reduce the
wavelet Parseval identity to the orthonormal basis identity for the basis $\{g_n(\cdot - 2^n \ell) :n\geq2,\ \ell \in\Z\}$.
The same construction allows the Calder\'on sum to be computed explicitly.

\subsection*{Notation}
The Fourier transform of a function $f \in L^1 (\R^d)$ is normalized as
$
\widehat f(\xi)
=
\int_{\R^d} f(x)e^{-2\pi i x\cdot \xi}\,dx,
$
where $x \cdot \xi$ denotes the standard dot product of vectors $x, \xi \in \R^d$.
We also denote by $\widehat{f}$ the usual extension of the Fourier transform to a function $f \in L^2(\R^d)$. The usual norm on $L^2 (\R^d)$ is denoted by $\| \cdot \|_2$.
We write $\mathcal F_x$
for the partial Fourier transform in the first variable. For
$x\in\R^d$, we write
$
T_x f=f(\,\cdot-x)
$
for translation of a function $f \in L^2 (\R^d)$. If $E\subseteq\R^d$ is measurable, we denote its
indicator function by $\mathds{1}_E$ and its Lebesgue measure by $|E|$.
For a full-rank lattice $\Lambda\subseteq\R^d$, its dual lattice is
$
\Lambda^*
=
\{\xi\in\R^d:\xi\cdot\lambda\in\Z
\text{ for all }\lambda\in\Lambda\}.
$
We write $A^t$ for the transpose of a matrix $A$.

\section{An orthonormal basis of translates} \label{sec:translation}
In this section, we construct the one-dimensional orthonormal basis needed
in the proof of Theorem~\ref{thm:main}. We first mention the following result of refinement of lattices, which is a special case of
\cite[Lemma~4.4]{fuehr2019system}.

\begin{lemma}[\cite{fuehr2019system}]  \label{lem:cosets}
Let $I\subseteq\{2,3,\ldots\}$ be infinite. For each $n \in I$, one can choose  $t_n \in \Z$,
 such that
\begin{equation}\label{eq:cosets}
\Z=\bigcupdot_{n\in I}(t_n+2^n\Z).
\end{equation}
\end{lemma}

Let
$\beta:\mathbb{N}_0 \to\Z$ be the bijection defined by
\[
\beta(0)=0,\qquad \beta(2r-1)=r,\qquad\beta(2r)=-r, \quad r\geq1.
\]
For $q\in\N$, put
$\nu_2(q)=\max\{r\in \mathbb{N}_0:2^r\text{ divides }q\}$, and define
\begin{equation}\label{eq:partition}
m_n := \beta\big(\nu_2(n-1)\big),\qquad
I_m := \{n\geq2:m_n=m\}, \quad m\in\Z.
\end{equation}
Then the sets $I_m$ are infinite and partition $\{2,3,\ldots\}$.
More precisely,
\begin{equation}\label{eq:partition_explicit}
I_{\beta(r)}=\{1+2^r(2q+1):q\in\mathbb{N}_0 \},\qquad r\in \mathbb{N}_0.
\end{equation}
In particular,
$
I_0=\{2,4,6,\ldots\} = 2\mathbb{N}
$.

The following proposition provides the orthonormal basis of translates for $L^2 (\R)$ 
that we will use in Section \ref{sec:wavelet} for the wavelet construction. 

\begin{proposition}\label{prop:translation}
There exist functions $g_n\in L^2(\R)$, $n \in \mathbb{N} \setminus \{1\}$, with
$\|g_n\|_2=1$ and $|\widehat g_n|\leq1$, such that 
the system of translates
$
\big \{T_{2^n \ell}g_n \big\}_{n \in \mathbb{N} \setminus \{1\}, \ell \in\Z }
$
is an orthonormal basis for $L^2(\R)$. Specifically, the functions $g_n$ can be chosen to have the form
\begin{equation}\label{eq:gn}
\widehat{g_n}(\omega )=e^{-2\pi i t_n \omega}\mathds{1}_{[m_n,m_n+1)}(\omega),
\qquad \omega \in\R,
\end{equation}
where $m_n$ is given by \eqref{eq:partition} and $t_n\in\Z$.
\end{proposition}

\begin{proof}
For each $m\in\Z$, apply Lemma~\ref{lem:cosets} to $I_m$ and choose
$t_n \in \Z$, $n\in I_m$, so that
\[
\Z=\bigcupdot_{n\in I_m}(t_n+2^n\Z).
\]
For $n \in \mathbb{N} \setminus \{1\}$, define $g_n$ by \eqref{eq:gn}. Then $\|g_n\|_2=1$ and
$|\widehat {g_n}|\leq1$. For $n\in I_m$ and $\ell \in\Z$, a direct calculation gives
\[
\widehat{T_{2^n \ell}g_n}(\omega)
=e^{-2\pi i(t_n+2^n \ell)\omega}\mathds{1}_{[m,m+1)}(\omega).
\]
Note that 
\[ \big\{ e^{-2\pi i(t_n+2^n \ell)\omega}\mathds{1}_{[m,m+1)} \big\}_{n \in I_m, \ell \in \Z} = 
\big\{e^{-2 \pi i k \omega} \mathds{1}_{[m, m+1)} \big\}_{k \in \mathbb{Z}} ,
\]
which is an orthonormal basis for $L^2([m,m+1))$. Therefore, it follows that the union $\cup_{m \in \Z} \big\{\widehat{T_{2^n \ell }g_n} \big\}_{n \in I_m, \ell \in \Z} = \{T_{2^n \ell}g_n \}_{n \in \mathbb{N} \setminus \{1\}, \ell \in\Z } $ is an orthonormal basis for $L^2 (\R)$. 
\end{proof}

\section{Construction of the orthonormal wavelet}\label{sec:wavelet}
Fix $a\geq1$ and put $A_a :=\diag(2a,1/2)$.
For $n\in \mathbb{N} \setminus \{1\}$, define
\begin{equation}\label{eq:En}
E_n := \left\{\xi\in\R:
\frac{(2a)^{1-n}}2\leq|\xi|<\frac{(2a)^{2-n}}2\right\}.
\end{equation}
Note that the sets form a measurable partition of $(-1/2,1/2) \setminus \{0\}$, and
\begin{equation}\label{eq:En_measure}
|E_n|=\big((2a)^2-2a\big)(2a)^{-n}.
\end{equation}
Define $\psi^{(a)} \in L^2 (\R^2)$ through its  Fourier transform via 
\begin{equation}\label{eq:fourier_psi}
\widehat{\psi^{(a)}}(\omega,\eta)
=\sum_{n=2}^{\infty}2^{-n/2}\mathds{1}_{E_n}(\omega)
\widehat g_n(2^{-n}\eta), \quad (\omega, \eta) \in \R^2,
\end{equation}
where the functions $g_n \in L^2 (\R)$ are as in Proposition~\ref{prop:translation}. Taking the inverse Fourier transform in the second variable shows that the partial Fourier transform of $\psi^{(a)}$ is
\begin{equation}\label{eq:partial_psi}
(\mathcal{F}_x \psi^{(a)})(\omega,y)
=\sum_{n=2}^{\infty}\mathds{1}_{E_n}(\omega)\,2^{n/2}g_n(2^ny), \quad (\omega, y) \in \R^2,
\end{equation}
Note that, for every $\omega \in \R$, at most one summand is nonzero. Moreover, we have that
$\|2^{n/2}g_n(2^n\cdot)\|_2=1$ for $n \in \mathbb{N} \setminus \{1\}$, cf. Proposition~\ref{prop:translation}. Therefore, using that  the sets $E_n$ form a partition of $(-1/2, 1/2) \setminus \{0\}$,
\begin{equation}\label{eq:psi_norm}
\|\mathcal{F}_x \psi^{(a)}\|_{2}^2
=\sum_{n=2}^{\infty}|E_n|
=1.
\end{equation}
Thus, the function
$\psi^{(a)}\in L^2(\R^2)$ is of unit norm. 

We next show that the associated wavelet system satisfies Parseval's identity, and thus forms an orthonormal basis.

\begin{theorem}\label{thm:parseval}
Let $a \geq 1$ and $A_a := \diag(2a, 1/2)$. Let $\psi^{(a)} \in L^2 (\R^2)$ be defined as in \eqref{eq:fourier_psi}. Then the wavelet system
$
\big\{ a^{-j/2} \psi^{(a)} (A_a^{-j} \cdot -  \gamma) \big \}_{j \in \mathbb{Z}, \gamma \in  \mathbb{Z}^2}
$
is an orthonormal basis for
$L^2(\R^2)$ satisfying
\[
\sum_{j\in\Z}
\big|\widehat{\psi^{(a)}}((A_a^t)^j\xi)\big|^2
\leq\frac13
\quad\text{for a.e.}\quad \xi\in\R^2.
\]
\end{theorem}

\begin{proof}
We split the proof into two steps.
\\~\\
\textbf{Step 1.}
Throughout this step, let $f\in L^2(\R^2)$ and put $F := \mathcal{F}_x f$. 
For $j,k,\ell \in\Z$, we write
\begin{align*}\label{eq:wavelet_atoms}
\psi^{(a)}_{j,k,\ell}(x,y) &=
|\det(A_a)|^{-j/2} \psi^{(a)} \bigg(A_a^{-j} \begin{pmatrix}
x \\ y
\end{pmatrix}
- 
\begin{pmatrix}
k \\ \ell
\end{pmatrix}
\bigg) \\
&=a^{-j/2}\psi^{(a)}\big((2a)^{-j}x-k,\,2^jy-\ell\big)
\end{align*}
for $(x,y) \in \R^2$. A change of variables in the first
coordinate gives
\begin{equation}\label{eq:partial_atom}
(\mathcal{F}_x \psi^{(a)}_{j,k,\ell})(\omega,y)
=(2a)^j a^{-j/2}e^{-2\pi i k(2a)^j\omega}
(\mathcal{F}_x \psi^{(a)})\big((2a)^j\omega,2^jy-\ell \big).
\end{equation}
Therefore, by Plancherel's theorem,
\begin{align*}
\langle f,\psi^{(a)}_{j,k,\ell}\rangle
=(2a)^j a^{-j/2}
\int_{\R}\int_{\R}
F(\omega,y)
\overline{(\mathcal{F}_x\psi^{(a)})((2a)^j\omega,2^jy-\ell)}
e^{2\pi i k(2a)^j\omega}\,dy\,d\omega.
\end{align*}

For fixed $j,\ell\in\Z$, set
\begin{equation}\label{eq:H}
H_{j,\ell}(\omega) := \int_{\R}F(\omega,y)
\overline{(\mathcal{F}_x \psi^{(a)} )\big((2a)^j\omega,2^jy-\ell \big)}\dif y.
\end{equation}
This integral is well defined for almost every $\omega$ and vanishes outside the interval
\[
J_j :=\left(-\frac{(2a)^{-j}}2,\frac{(2a)^{-j}}2\right),
\]
because $(\mathcal{F}_x \psi^{(a)}) ((2a^j) \omega, 2^j - \ell) = 0$ when $(2a)^j \omega \in (-1/2, 1/2) \setminus \{0\}$, while, for almost every $\omega \in J_j \setminus \{0\}$, there is a unique $n\geq2$ such that
$(2a)^j\omega\in E_n$. Furthermore, note that \eqref{eq:partial_psi} implies
\[
\big\|\mathcal{F}_x \psi^{(a)} \big((2a)^j\omega,2^j\cdot-\ell \big)\big\|_2^2
=2^n\int_{\R}|g_n(2^{n+j}y-2^n \ell)|^2\dif y=2^{-j}.
\]
Hence, an application of the Cauchy-Schwarz inequality yields
\begin{equation}\label{eq:H_estimate}
|H_{j,\ell}(\omega)|^2\leq2^{-j}\mathds{1}_{J_j}(\omega)
\int_{\R}|F(\omega,y)|^2\dif y.
\end{equation}
In particular, $H_{j,\ell}\in L^2(J_j)\subseteq L^1(J_j)$.

An application of Plancherel's theorem and Fubini's theorem allows us therefore to write
\[
\langle f,\psi^{(a)}_{j,k,\ell}\rangle
=(2a)^j a^{-j/2}\int_{J_j}H_{j,\ell}(\xi)
 e^{2\pi i k(2a)^j\xi}\dif\xi.
\]
Note that the system $\big\{ (2a)^{j/2} e^{-2\pi i k(2a)^j\omega} \big\}_{k\in \Z}$ forms an
orthonormal basis for $L^2(J_j)$. Hence,
an application of Parseval's identity gives
\begin{align}\label{eq:parseval_k}
\sum_{k\in\Z}|\langle f,\psi^{(a)}_{j,k,\ell}\rangle|^2
=(2a)^{2j}a^{-j}(2a)^{-j}
\int_{\R}|H_{j,\ell}(\xi)|^2\dif\omega 
=2^j\int_{\R}|H_{j,\ell}(\omega)|^2\dif\omega.
\end{align}

For fixed $\omega\neq0$, let $s=s(\omega)\in\Z$ be the unique integer
such that
\begin{equation}\label{eq:shell}
\frac{2a}{2}\leq(2a)^s|\omega|<\frac{(2a)^2}{2}.
\end{equation}
By \eqref{eq:En}, for $j\in\Z$ and $n\geq2$,
we have that
$(2a)^j\omega\in E_n
$ if and only if $ j+n=s$.
Thus, $H_{j,\ell}(\omega)=0$ when $j>s-2$.  
For
$j\leq s-2$, set $n:=s-j \in \mathbb{N} \setminus \{1\}$, so that $(2a)^j \omega \in E_n$.
Substituting \eqref{eq:partial_psi} into \eqref{eq:H}, and using
$n+j=s$, yields
\[
H_{j,\ell}(\omega)
=2^{n/2}\int_{\R}F(\omega,y)
\overline{g_n(2^{n+j}y-2^n \ell)}\,dy
=2^{n/2}\int_{\R}F(\omega,y)
\overline{g_n(2^sy-2^n\ell)}\,dy.
\]
The change of variable  $u=2^sy$ next yields
\begin{align*}
H_{j,\ell}(\omega)
&=2^{n/2-s}\int_{\R}F(\omega,2^{-s}u)
\overline{g_n(u-2^n \ell)}\,du
=2^{(n-s)/2}\int_{\R}F_\omega (u)
\overline{g_n(u-2^n \ell)}\,du,
\end{align*}
where $F_\omega \in L^2(\R)$, for almost all $\omega \in \R$, is defined by $F_\omega(u) :=2^{-s/2}F(\omega,2^{-s}u)$ for $u \in \R$. Note that
\begin{equation}\label{eq:fiber_norm}
\|F_\omega \|_2^2
=2^{-s}\int_{\R}|F(\omega,2^{-s}u)|^2\,du
=\int_{\R}|F(\omega,y)|^2\,dy.
\end{equation} 
Thus, the above shows that
$
H_{j,\ell} (\omega) = 2^{-j/2}\langle F_\omega,T_{2^n \ell}g_n\rangle $.
As $j$ ranges over the integers satisfying $j\leq s-2$,
the index $n=s-j$ ranges over all integers $n\geq2$.
Consequently, using that $H_{j,\ell} (\omega) = 0$ for $j > s-2$,
\begin{align}
\sum_{j,\ell\in\Z}2^j|H_{j,\ell}(\omega)|^2
&=\sum_{j\leq s-2}\sum_{\ell\in\Z}
|\langle F_\omega,T_{2^{s-j}\ell}g_{s-j}\rangle|^2\notag
=\sum_{n=2}^{\infty}\sum_{\ell\in\Z}
|\langle F_\omega,T_{2^n\ell}g_n\rangle|^2 \\
&=\|F_\omega\|_2^2,
\label{eq:parseval_fiber}
\end{align}
where the last step used Parseval's identity for the orthonormal
basis $\{T_{2^n \ell}g_n \}_{n \in \N \setminus \{1\}, \ell \in \mathbb{N}}$, cf. Proposition~\ref{prop:translation}.

Lastly, a combination of \eqref{eq:parseval_k},   \eqref{eq:parseval_fiber} and \eqref{eq:fiber_norm} yields
\begin{align*}
\sum_{j,k,\ell\in\Z}
|\langle f,\psi^{(a)}_{j,k,\ell}\rangle|^2
&=\int_{\R}\sum_{j,\ell\in\Z}2^j|H_{j,\ell}(\omega)|^2\,d\omega =\int_{\R}\|F_\omega \|_2^2\,d\omega =\int_{\R}\int_{\R}|F(\omega,y)|^2\,dy\,d\omega \\
&=\|f\|_{2}^2.
\end{align*}
This shows that $
\big\{ a^{-j/2} \psi^{(a)} (A_a^{-j} \cdot -  \gamma) \big \}_{j \in \mathbb{Z}, \gamma \in  \mathbb{Z}^2}
$ is a Parseval frame for $L^2 (\R^2)$. Since $\psi^{(a)} \in L^2 (\R^2)$ has norm one, it follows that it must be an orthonormal basis.
\\~\\
\textbf{Step 2.} We next show the claim on the Calder\'on sum. For
$(\omega,\eta)\in\R^2$ with $\omega\neq0$, let $s=s(\omega) \in \Z$ be given
by \eqref{eq:shell}. As above, by \eqref{eq:En}, the condition
$(2a)^j\omega\in E_n$ is equivalent to $j+n=s$. Thus, if $j > s-2$, then 
$
\widehat{\psi^{(a)}}((2a)^j\omega,2^{-j}\eta)=0
$.
For $j\leq s-2$, set $n:=s-j\geq2$. Then
\eqref{eq:fourier_psi} gives
\[
\widehat{\psi^{(a)}}((2a)^j\omega,2^{-j}\eta)
=
2^{-n/2}\widehat g_n(2^{-n}2^{-j}\eta)
=
2^{-n/2}\widehat g_n(2^{-s}\eta).
\]
Therefore,
\begin{align*}
\sum_{j\in\Z}
\big|\widehat{\psi^{(a)}}((2a)^j\omega,2^{-j}\eta)\big|^2
&=
\sum_{n=2}^{\infty}
2^{-n}|\widehat g_n(2^{-s}\eta)|^2 \\
&=
\sum_{n=2}^{\infty}
2^{-n}\mathds{1}_{[m_n,m_n+1)}(2^{-s}\eta),
\end{align*}
where the last equality follows from \eqref{eq:gn}.

Lastly, let $m = m(s, \eta) \in\Z$ be the unique integer such that
$2^{-s}\eta\in[m,m+1)$. Then the above identity shows that
\[ \sum_{j \in \Z} |\widehat{\psi^{(a)}} \big((A^t_a)^j 
\begin{pmatrix}
\omega \\
\eta
\end{pmatrix} \big) |^2 = 
\sum_{j\in\Z}
\big|\widehat{\psi^{(a)}}((2a)^j\omega,2^{-j}\eta)\big|^2 = \sum_{n\in I_m}2^{-n}.
\]
On the one hand, if $m=0$, then $I_0=2 \mathbb{N}$, so that
\[
\sum_{n\in I_0}2^{-n}
=\sum_{r=1}^{\infty}2^{-2r}
=\frac13.
\]
On the other hand, if $m\neq0$, then $I_m$ consists of odd integers greater than or equal to $3$,
and therefore
\[
\sum_{n\in I_m}2^{-n}
\leq\sum_{r=1}^{\infty}2^{-(2r+1)}
=\frac16.
\]
In combination, this shows the claim.
\end{proof}

Lastly, we explain that the pair $(A_a, \Z^2)$ does not admit a wavelet set. For $a>1$, we use the
two-dimensional characterization of Ionascu and Wang
\cite[Theorem~1.3]{ionascu2006simultaneous}, in the form stated in
\cite[Theorem~1.3]{bownik2026simultaneous}: if $A\in\mathrm{GL}(2,\R)$ has
$|\det(A)|>1$ and eigenvalues $\lambda_1,\lambda_2$ with
$|\lambda_1|\geq|\lambda_2|$, then $(A,\Z^2)$ admits a wavelet set if and
only if either $|\lambda_2|\geq1$, or
\[
|\lambda_2|<1\quad\text{and}\quad
\ker(A-\lambda_2 I)\cap\Z^2 =\{0\}.
\]

This easily gives the result.

\begin{proposition}\label{prop:no_wavelet_set}
For each $a \geq 1$, the pair $(A_a,\Z^2)$ does not admit a wavelet set.
\end{proposition}

\begin{proof}
If $a = 1$, then $\det(A_a) = 1$, and hence it cannot admit a wavelet set by, e.g., \cite[Theorem 4]{larson2006explicit}.
If $a>1$, then the eigenvalues of $A_a$ are $\lambda_1 = 2a$ and $\lambda_2 = 1/2$, and
\[
\ker\big(A_a -\tfrac12 I\big)\cap\Z^2=\{0\}\times\Z\neq\{0\}.
\]
Hence, \cite[Theorem~1.3]{bownik2026simultaneous} implies that $(A_a,\Z^2)$ does not admit a
wavelet set.
\end{proof}

A combination of Theorem~\ref{thm:parseval} and Proposition \ref{prop:no_wavelet_set} directly yields Theorem \ref{thm:main}.

\section*{Tool and computational resource disclosure}
Large language models (LLMs), specifically ChatGPT 5.6 Sol and ChatGPT 6 Astra, played a substantial role in the development of this work. The authors had previously considered adapting the known counterexamples to Calder\'on-type identities for generalized shift-invariant systems \cite{bownik2004spectral} to the setting of wavelets, but had not succeeded in carrying out this adaptation.
The LLMs were used to turn this idea into the explicit wavelet construction presented here.
The initial construction led to a wavelet orthonormal basis for the dilation matrix $\diag(3, 1/3)$ of determinant one. For such a matrix, neither a function satisfying the
Calder\'on sum formula nor a wavelet set can exist, so this already provided a counterexample to both Conjecture \ref{conj:wavelet} and Conjecture \ref{con:waveletset}. We then used LLMs to investigate whether this initial example could be adapted to the case $|\det(A)| \neq 1$, which led to the present Theorem \ref{thm:main}. 

The authors have independently verified, validated, and rewritten all portions of the paper influenced by LLM-generated material, and assume full responsibility for the mathematical content.
 
 \section*{Acknowledgements}
 For J.v.V., this research was funded in whole or in part by the Austrian Science Fund (FWF): 10.55776/PAT2545623. For open access purposes, the author has applied a CC BY public copyright license to any author-accepted manuscript version arising from this submission.  F.V. acknowledges support by the German Science Foundation (DFG) in the context of the Emmy Noether junior
research group VO 2594/1-1 and the Hightech Agenda Bavaria.

\bibliographystyle{abbrv}
\bibliography{bibl}

@article{rzeszotnik2001calderon,
 author = {Rzeszotnik, Ziemowit},
 title = {Calder{\'o}n's condition and wavelets},
 fjournal = {Collectanea Mathematica},
 journal = {Collect. Math.},
 issn = {0010-0757},
 volume = {52},
 number = {2},
 pages = {181--191},
 year = {2001},
 language = {English},
 url = {https://eudml.org/doc/42731},
 zbMATH = {1655778},
 Zbl = {0989.42017}
}

@book{dai1998wandering,
 author = {Dai, Xingde and Larson, David R.},
 title = {Wandering vectors for unitary systems and orthogonal wavelets},
 fseries = {Memoirs of the American Mathematical Society},
 series = {Mem. Am. Math. Soc.},
 issn = {0065-9266},
 volume = {640},
 isbn = {978-0-8218-0800-9; 978-1-4704-0229-7},
 year = {1998},
 publisher = {Providence, RI: American Mathematical Society (AMS)},
 language = {English},
 doi = {10.1090/memo/0640},
 zbMATH = {1212343},
 Zbl = {0990.42022}
}

@incollection{larson2006explicit,
 author = {Larson, David and Schulz, Eckart and Speegle, Darrin and Taylor, Keith F.},
 title = {Explicit cross-sections of singly generated group actions},
 booktitle = {Harmonic analysis and applications. In Honor of John J. Benedetto},
 isbn = {0-8176-3778-8},
 pages = {209--230},
 year = {2006},
 publisher = {Basel: Birkh{\"a}user},
 language = {English},
 zbMATH = {5064380},
 Zbl = {1129.42437}
}

@incollection{bownik2004spectral,
 author = {Bownik, Marcin and Rzeszotnik, Ziemowit},
 title = {The spectral function of shift-invariant spaces on general lattices},
 booktitle = {Wavelets, frames and operator theory. Papers from the Focused Research Group Workshop, University of Maryland, College Park, MD, USA, January 15--21, 2003},
 isbn = {0-8218-3380-4},
 pages = {49--59},
 year = {2004},
 publisher = {Providence, RI: American Mathematical Society (AMS)},
 language = {English},
 zbMATH = {2144538},
 Zbl = {1083.42025}
}

@article{wang2002wavelets,
 author = {Wang, Yang},
 title = {Wavelets, tiling, and spectral sets},
 fjournal = {Duke Mathematical Journal},
 journal = {Duke Math. J.},
 issn = {0012-7094},
 volume = {114},
 number = {1},
 pages = {43--57},
 year = {2002},
 language = {English},
 doi = {10.1215/S0012-7094-02-11413-6},
 zbMATH = {1820918},
 Zbl = {1011.42024}
}

@article{ionascu2006simultaneous,
 author = {Ionascu, Eugen J. and Wang, Yang},
 title = {Simultaneous translational and multiplicative tiling and wavelet sets in $\mathbb{R}^2$},
 fjournal = {Indiana University Mathematics Journal},
 journal = {Indiana Univ. Math. J.},
 issn = {0022-2518},
 volume = {55},
 number = {6},
 pages = {1935--1949},
 year = {2006},
 language = {English},
 doi = {10.1512/iumj.2006.55.2967},
 zbMATH = {5125666},
 Zbl = {1145.42008}
}

@article{fuehr2019system,
 author = {F{\"u}hr, Hartmut and Lemvig, Jakob},
 title = {System bandwidth and the existence of generalized shift-invariant frames},
 fjournal = {Journal of Functional Analysis},
 journal = {J. Funct. Anal.},
 issn = {0022-1236},
 volume = {276},
 number = {2},
 pages = {563--601},
 year = {2019},
 language = {English},
 doi = {10.1016/j.jfa.2018.10.014},
 zbMATH = {6988449},
 Zbl = {1407.42023}
}

@misc{enstad2026calderon,
 author = {Ulrik Enstad and Jordy Timo van Velthoven},
 title = {On the {Calder{\'o}n} sum formula for wavelet systems},
 year = {2026},
 howpublished = {Preprint, {arXiv}:2602.10766},
 url = {https://arxiv.org/abs/2602.10766},
 arXiv = {arXiv:2602.10766}
}

@article{bownik2026msf,
 author = {Bownik, Marcin and Rzeszotnik, Ziemowit and Speegle, Darrin},
 title = {Are {MSF} wavelets minimally supported?},
 fjournal = {Expositiones Mathematicae},
 journal = {Expo. Math.},
 issn = {0723-0869},
 volume = {44},
 number = {2},
 pages = {16},
 note = {Id/No 125768},
 year = {2026},
 language = {English},
 doi = {10.1016/j.exmath.2026.125768},
 zbMATH = {8174986}
}

@incollection{larson2007unitary,
 author = {Larson, David R.},
 title = {Unitary systems, wavelet sets, and operator-theoretic interpolation of wavelets and frames},
 booktitle = {Gabor and wavelet frames},
 isbn = {978-981-270-907-3},
 pages = {167--214},
 year = {2007},
 publisher = {Hackensack, NJ: World Scientific},
 language = {English},
 zbMATH = {5222516},
 Zbl = {1134.42340}
}

@article{bownik2026simultaneous,
 author = {Bownik, Marcin and Speegle, Darrin},
 title = {Simultaneous dilation and translation tilings of {{\(\mathbb{R}^n\)}}},
 fjournal = {American Journal of Mathematics},
 journal = {Am. J. Math.},
 issn = {0002-9327},
 volume = {148},
 number = {2},
 pages = {473--503},
 year = {2026},
 language = {English},
 doi = {10.1353/ajm.2026.a986598},
 zbMATH = {8183842}
}

@article{laugesen2002translational,
 author = {Laugesen, Richard S.},
 title = {Translational averaging for completeness, characterization and oversampling of wavelets.},
 fjournal = {Collectanea Mathematica},
 journal = {Collect. Math.},
 issn = {0010-0757},
 volume = {53},
 number = {3},
 pages = {211--249},
 year = {2002},
 language = {English},
 url = {https://eudml.org/doc/42924},
 zbMATH = {1943355},
 Zbl = {1035.42035}
}

@article{calogero2000characterization,
 author = {Calogero, A.},
 title = {A characterization of wavelets on general lattices.},
 fjournal = {The Journal of Geometric Analysis},
 journal = {J. Geom. Anal.},
 issn = {1050-6926},
 volume = {10},
 number = {4},
 pages = {597--622},
 year = {2000},
 language = {English},
 doi = {10.1007/BF02921988},
 zbMATH = {1679938},
 Zbl = {1057.42025}
}

@article{bownik2000characterization,
 author = {Bownik, Marcin},
 title = {A characterization of affine dual frames in {{\(L^2(\mathbb{R}^n)\)}}},
 fjournal = {Applied and Computational Harmonic Analysis},
 journal = {Appl. Comput. Harmon. Anal.},
 issn = {1063-5203},
 volume = {8},
 number = {2},
 pages = {203--221},
 year = {2000},
 language = {English},
 doi = {10.1006/acha.2000.0284},
 zbMATH = {1437487},
 Zbl = {0961.42018}
}

@article{chui2002characterization,
 author = {Chui, Charles K. and Czaja, Wojciech and Maggioni, Mauro and Weiss, Guido},
 title = {Characterization of general tight wavelet frames with matrix dilations and tightness preserving oversampling},
 fjournal = {The Journal of Fourier Analysis and Applications},
 journal = {J. Fourier Anal. Appl.},
 issn = {1069-5869},
 volume = {8},
 number = {2},
 pages = {173--200},
 year = {2002},
 language = {English},
 doi = {10.1007/s00041-002-0007-4},
 zbMATH = {1761903},
 Zbl = {1005.42020}
}

@article{hernandez2002unified,
 author = {Hern{\'a}ndez, Eugenio and Labate, Demetrio and Weiss, Guido},
 title = {A unified characterization of reproducing systems generated by a finite family. {II}},
 fjournal = {The Journal of Geometric Analysis},
 journal = {J. Geom. Anal.},
 issn = {1050-6926},
 volume = {12},
 number = {4},
 pages = {615--662},
 year = {2002},
 language = {English},
 doi = {10.1007/BF02930656},
 url = {www.lib.ncsu.edu/resolver/1840.2/495},
 zbMATH = {1983423},
 Zbl = {1039.42032}
}

@article{speegle2003existence,
 author = {Speegle, Darrin},
 title = {On the existence of wavelets for non-expansive dilation matrices},
 fjournal = {Collectanea Mathematica},
 journal = {Collect. Math.},
 issn = {0010-0757},
 volume = {54},
 number = {2},
 pages = {163--179},
 year = {2003},
 language = {English},
 url = {https://eudml.org/doc/43075},
 zbMATH = {2162763},
 Zbl = {1062.42030}
}

@article{laugesen2002characterization,
 author = {Laugesen, R. S. and Weaver, N. and Weiss, G. L. and Wilson, E. N.},
 title = {A characterization of the higher dimensional groups associated with continuous wavelets.},
 fjournal = {The Journal of Geometric Analysis},
 journal = {J. Geom. Anal.},
 issn = {1050-6926},
 volume = {12},
 number = {1},
 pages = {89--102},
 year = {2002},
 language = {English},
 doi = {10.1007/BF02930862},
 zbMATH = {1777197},
 Zbl = {1043.42032}
}

@article{bownik2017wavelets,
 author = {Bownik, Marcin and Lemvig, Jakob},
 title = {Wavelets for non-expanding dilations and the lattice counting estimate},
 fjournal = {IMRN. International Mathematics Research Notices},
 journal = {Int. Math. Res. Not.},
 issn = {1073-7928},
 volume = {2017},
 number = {23},
 pages = {7264--7291},
 year = {2017},
 language = {English},
 doi = {10.1093/imrn/rnw222},
 zbMATH = {7004503},
 Zbl = {1405.42059}
}

@incollection{bownik2020open,
 author = {Bownik, Marcin and Rzeszotnik, Ziemowit},
 title = {Open problems in wavelet theory},
 booktitle = {Operator theory, operator algebras and their interactions with geometry and topology. Ronald G. Douglas memorial volume. Proceedings of the international workshop on operator theory and applications (IWOTA 2018), Shanghai, China, July 23--27, 2018},
 isbn = {978-3-030-43379-6; 978-3-030-43382-6; 978-3-030-43380-2},
 pages = {77--100},
 year = {2020},
 publisher = {Cham: Birkh{\"a}user},
 language = {English},
 doi = {10.1007/978-3-030-43380-2_4},
 zbMATH = {7299674},
 Zbl = {1459.42050}
}

\end{document}